\documentclass[letterpaper, 10 pt, conference, twocolumn]{ieeeconf}  

\IEEEoverridecommandlockouts                              
\usepackage{graphicx}
\twocolumn
\usepackage{graphicx}%
\usepackage{multirow}%
\usepackage{amsmath,amssymb,amsfonts}%
\usepackage{mathrsfs}%
\usepackage{xcolor}%
\usepackage{textcomp}%
\usepackage{manyfoot}%
\usepackage{booktabs}%
\usepackage{algorithm}%
\usepackage{algorithmicx}%
\usepackage{algpseudocode}%
\usepackage{listings}%
\usepackage{tabularx}
\usepackage{subcaption}
\usepackage{mathtools}
\usepackage{nicefrac, xfrac}

\usepackage[bb=dsserif]{mathalpha}
\usepackage{makecell}

\usepackage{pifont}
\newcommand{\n}{{\rho}}
\newcommand{\m}{{\eta}}
\usepackage[utf8]{inputenc} 
\usepackage[T1]{fontenc}    
\usepackage{hyperref}       
\usepackage{url}            
\usepackage{booktabs}       
\usepackage{amsfonts}       
\usepackage{nicefrac}       
\usepackage{microtype}      
\usepackage{xcolor}         
\usepackage{bm}
\usepackage{ascii}
\usepackage{newunicodechar}
\newunicodechar{☺}{\SOH}
\newunicodechar{☻}{\STX}
\newunicodechar{♥}{\ETX}
\newunicodechar{♣}{\ENQ}
\newunicodechar{♠}{\ACK}
\newunicodechar{○}{\HT}

\makeatletter
\newsavebox\myboxA
\newsavebox\myboxB
\newlength\mylenA
\newcommand*\xoverline[2][0.75]{%
    \sbox{\myboxA}{$\m@th#2$}%
    \setbox\myboxB\null
    \ht\myboxB=\ht\myboxA%
    \dp\myboxB=\dp\myboxA%
    \wd\myboxB=#1\wd\myboxA
    \sbox\myboxB{$\m@th\overline{\copy\myboxB}$}
    \setlength\mylenA{\the\wd\myboxA}
    \addtolength\mylenA{-\the\wd\myboxB}%
    \ifdim\wd\myboxB<\wd\myboxA%
       \rlap{\hskip 0.5\mylenA\usebox\myboxB}{\usebox\myboxA}%
    \else
        \hskip -0.5\mylenA\rlap{\usebox\myboxA}{\hskip 0.5\mylenA\usebox\myboxB}%
    \fi}
\makeatother
\usepackage{xparse}
\NewDocumentCommand{\Iij}{ O{i} O{j} }{\ensuremath{[#1>#2]}}
\NewDocumentCommand{\Iji}{ O{i} O{j} }{\ensuremath{[#1<#2]}}
\NewDocumentCommand{\TigLmu}{ O{k} O{\ensuremath{\gamma L}} O{\ensuremath{\kappa}} }{\ensuremath{T_{#1} (#2, #3)}}
\NewDocumentCommand{\Ei}{ O{k} O{x} }{\ensuremath{E_{#1} (#2)}}
\NewDocumentCommand{\Tinm}{ O{k} }{\ensuremath{T_{#1}(\n,\m)}}
\NewDocumentCommand{\gbari}{ O{k} O{\ensuremath{(\kappa})} }{\ensuremath{\xoverline{\gamma L}_{#1}{#2}}}

\NewDocumentCommand{\gstar}{ O{\ensuremath{\kappa}} }{\ensuremath{(\gamma L)^*{#1}}}
\NewDocumentCommand{\pigLgmu}{ O{\ensuremath{\gamma}} }{\ensuremath{p({#1}L, {#1}\mu)}}
\NewDocumentCommand{\idproof}{ O{\ensuremath{k}} }{\ensuremath{#1}}
\NewDocumentCommand{\thetagLgmu}{ O{\ensuremath{\gamma}} }{\ensuremath{\theta(#1 L, #1 \mu)}}
\NewDocumentCommand{\cgLgmu}{ O{\ensuremath{\gamma}} }{\ensuremath{c(#1 L, #1 \mu)}}

\usepackage{comment}
\newcommand{\centered}[1]{\begin{tabular}{l} #1 \end{tabular}}

\usepackage{svg}
\graphicspath{{./pics/}}
\DeclareMathOperator*{\argmin}{argmin}

\DeclareMathOperator*{\maximize}{maximize}
\DeclareMathOperator*{\minimize}{minimize}

\AtBeginEnvironment{appendices}{\crefalias{section}{appendix}}

\usepackage{varioref}
\usepackage{hyperref}
\usepackage[noabbrev, capitalise, nameinlink]{cleveref}

\makeatletter
\let\cl@chapter\undefined
\makeatletter
\newtheorem{theorem}{Theorem}[section]
\newtheorem{proposition}{Proposition}[section]
\newtheorem{conjecture}{Conjecture}[section]

\newtheorem{corollary}{Corollary}[section]
\newtheorem{remark}{Remark}[section]
\newtheorem{definition}{Definition}[section]
\newtheorem{assumption}{Assumption}[section]
\newtheorem{lemma}{Lemma}[section]

\hypersetup{%
      breaklinks,
      hypertexnames = true,
      plainpages    = false,
      colorlinks = true,
      urlcolor   = cyan,
      linkcolor  = blue,
      filecolor  = magenta,      
      citecolor  = brown,
    }

\title{\LARGE \bf
Improved convergence rates for the Difference-of-Convex algorithm}

\author{Teodor Rotaru, Panagiotis Patrinos, François Glineur
\thanks{T. Rotaru and P. Patrinos are with Department of Electrical Engineering (ESAT-STADIUS),  KU Leuven, B-3001 Leuven, Belgium
        {\tt\scriptsize teodor.rotaru@kuleuven.be; panos.patrinos@kuleuven.be}}%
\thanks{T. Rotaru and F. Glineur are with ICTEAM-INMA and CORE,  Université catholique de Louvain, B-1348 Louvain-la-Neuve, Belgium
        {\tt\scriptsize francois.glineur@uclouvain.be}}%
\thanks{T. Rotaru is supported by a grant from the Global PhD partnership between KU Leuven and UCLouvain.}
}

\begin{document}
\twocolumn

\maketitle
\thispagestyle{empty}
\pagestyle{empty}

\begin{abstract}%
We consider a difference-of-convex formulation where one of the terms is allowed to be hypoconvex (or weakly convex). We first examine the precise behavior of a single iteration of the Difference-of-Convex algorithm (DCA), giving a tight characterization of the objective function decrease. This requires distinguishing between eight distinct parameter regimes.

Our proofs are inspired by the performance estimation framework, but are much simplified compared to similar previous work.

We then derive sublinear DCA convergence rates towards critical points, distinguishing between cases where at least one of the functions is smooth and where both functions are nonsmooth. We conjecture the tightness of these rates for four parameter regimes, based on strong numerical evidence obtained via performance estimation, as well as the leading constant in the asymptotic sublinear rate for two more regimes.
\end{abstract}

\section{Introduction}
Consider the difference-of-convex formulation
\begin{align}\label{eq:DC_program}\tag{DC}
    \minimize_{x \in \mathbb{R}^d} F(x) := f_1(x) - f_2(x),
\end{align}
where $f_1, f_2: \mathbb{R}^d \rightarrow \mathbb{R}$ are proper, lower semicontinuous convex functions, and $F$ is lower bounded. 

A standard method to solve \eqref{eq:DC_program} is the Difference-of-Convex Algorithm (DCA), a versatile method with no parameter that can find a \emph{critical} point of $F$, defined as a point $x^* \in \mathbb{R}^d$ for which there exists subgradients $g_1^* \in \partial 
f_1(x^*)$ and $g_2^* \in \partial 
f_2(x^*)$ such that $g_1^*=g_2^*$. 
Stationary points of $F$ are always critical, but the converse is not true.
Extensive analyses of DCA are provided in \cite{Dinh_Thi_dca_1997,LeThi_2018_30_years_dev,Horst_Thoai_DC_overview_1999}. DCA is also referred to as the convex-concave procedure (CCCP), as seen in \cite{CCCP_2001_init_Alan_Anand,CCCP_2009_Convergence,CCCP_Lipp_Boyd_2016}.

Convergence rates of DCA to critical points can be derived in two distinct cases, each using a suitable performance criterion: case (S) when at least one of $f_1$ and $f_2$ is smooth, and case (NS) 
when both $f_1$ and $f_2$ are nonsmooth. The recent work \cite{abbaszadehpeivasti2021_DCA} by Abbaszadehpeivasti, de Klerk and Zamani provides exact convergence rates of DCA in both cases. Their approach, based on performance estimation \cite{drori_performance_2014, taylor_smooth_2017}, leads to rigorous proofs for those rates. The fact that these rates are exact is supported by strong numerical
evidence and, in some cases, by the identification of instances matching those rates exactly.

In this work, we generalize the standard \eqref{eq:DC_program} setting and consider the case where one of the terms can be nonconvex. More precisely, we assume that either $f_1$ or $f_2$ is hypoconvex (or weakly convex). Some previous works also introduce hypoconvexity in either $f_1$ (see  \cite{DC_alg_dif_Moreau_Smoothing_2022}) or $f_2$ (see \cite{dc_weakly_cvx_bundle_method_2023_Syrtseva}).

Our results follow the same line as \cite{abbaszadehpeivasti2021_DCA}, also relying on performance estimation. More precisely:
\begin{itemize}\item In the smooth case (S), we characterize in \cref{thm:one_step_decrease_dca} the exact behavior of one iteration of \eqref{eq:DCA_it} in the setting where $f_1$ or $f_2$ is hypoconvex, providing a lower bound for the decrease of the objective expressed in terms of differences of subgradients. 

\item \cref{thm:one_step_decrease_dca} describes a total of eight different regimes, partitioning the space of problem parameters (in terms of smoothness and strong convexity of both terms, see below) in eight disjoint regions. We conjecture that these bounds on the objective decrease are tight for all of those eight regimes.

Among those eight regimes only two  were previously known and proved in \cite{abbaszadehpeivasti2021_DCA}, corresponding to the standard \eqref{eq:DCA_it} setting ($f_1$ and $f_2$ convex) where in addition $F$ is required to be both nonconvex and nonconcave.
\item We prove in \cref{thm:dca_rates_N_steps} that, in our setting allowing one hypoconvex term, DCA converges sublinearly to critical points, with a  $\mathcal{O}(\frac{1}{N})$ rate after $N$ iterations, again with eight distinct regimes.
\item Based on strong numerical evidence, we conjecture that four of those rates are exact for any number of iterations, while for two other regimes we identify the leading constant in the sublinear rate.
\item In the nonsmooth case (NS), we characterize the decrease of the objective function after one step in \cref{thm:both_functions_nonsmooth}, using a suitably adapted optimality criterion, and deduce the corresponding $\mathcal{O}(\frac{1}{N})$ sublinear rate in  \cref{thm:convergence_both_nonsmooth}, strengthening and generalizing  the existing rates of \cite{abbaszadehpeivasti2021_DCA}. 
\end{itemize}

\section{Theoretical background}%

\begin{definition}\label{def:upper_lower_curvature} 
Let $L>0$ and $\mu \le L$. We say that a proper, lower semicontinuous function $f:\mathbb{R}^d\rightarrow \mathbb{R}$ belongs to the class $\mathcal{F}_{\mu, L}(\mathbb{R}^d)$ (or simply $\mathcal{F}_{\mu, L}$) if and only if it has both (i) upper curvature $L$, meaning that $ \tfrac{L}{2} \|\cdot\|^2 - f$ is convex, and (ii) lower curvature $\mu$, meaning that $ f - \tfrac{\mu}{2} \|\cdot\|^2$ is convex. We also define the class $\mathcal{F}_{\mu, \infty}(\mathbb{R}^d)$ which requires only lower curvature $\mu$.
\end{definition}%

Intuitively, the curvature bounds $\mu$ and $L$ correspond to the minimum and maximum eigenvalues of the Hessian for a function $f \in \mathcal{C}^2$. Functions in $\mathcal{F}_{\mu, L}$ must be smooth when $L<\infty$, while $\mathcal{F}_{\mu, \infty}$ also contains nonsmooth functions.

Depending on the sign of the lower curvature $\mu$, a function $f \in \mathcal{F}_{\mu,L}$ is categorized as: (i) hypoconvex (or weakly convex) when $\mu<0$, (ii) convex when $\mu=0$ or (iii) strongly convex for $\mu>0$. 

\smallskip The subdifferential of a proper, lower semicontinuous convex function $f:\mathbb{R}^d \rightarrow \mathbb{R}$ at a point $x \in \mathbb{R}^d$ is defined as
\begin{align*}\label{eq:subdif_cvx}
    \partial f(x) := \{ g \in \mathbb{R}^d \,|\, f(y) {}\geq{} f(x) + \langle g, y-x \rangle  \ \forall y\in \mathbb{R}^d\}.
\end{align*}%
For hypoconvex functions the subdifferential can be defined as follows \cite[Proposition 6.3]{Bauschke_generalized_monotone_operators_2021}. Let $f \in \mathcal{F}_{\mu,\infty}$ be a hypoconvex function with $\mu < 0$. Then $\tilde{f}(x):=f(x) - \mu \tfrac{\|x\|^2}{2}$ is convex with a well-defined subdifferential $\tilde{f}(x)$, and we let \[ \partial f(x) := \{ g - \mu x \,|\, g \in \partial \tilde{f}(x) \}.\]
If $f$ is differentiable at $x$, then $\partial f(x) = \{ \nabla f(x)\}$.

\smallskip The following Lemma will prove useful later.
\begin{lemma}[Lemma 2.5 from \cite{rotaru2022tight}]\label{lemma:hypoconvex_smooth_quad_bounds}
    Let $L > 0$ and $\mu \le L$, and let $f\in \mathcal{F}_{\mu, L}$.  Then we have the following $\forall x, y \in \mathbb{R}^d$:
$$
    \tfrac{\mu}{2} \|x-y\|^2 
        {} \leq {} 
    f(x) - f(y) - \langle g, x-y \rangle 
        {} \leq {} 
    \tfrac{L}{2} \|x-y\|^2,
$$
where $g \in \partial f(y)$.
\end{lemma}%

\smallskip \begin{assumption}[Objective and parameters]\label{assumption:curvatures_on_F}
The objective function $F$ in \eqref{eq:DC_program} is lower bounded and can be written $F = f_1 - f_2$ where $f_1 \in \mathcal{F}_{\mu_1,L_1}$ and $f_2 \in \mathcal{F}_{\mu_2,L_2}$ with parameters $L_1,L_2 \in (0,\infty]$ and $\mu_1, \mu_2 \in (-\infty,\infty)$ such that $\mu_1<L_1$ and $\mu_2<L_2$.
\end{assumption}

We make \cref{{assumption:curvatures_on_F}} through the rest of this paper, and observe that it implies $F \in \mathcal{F}_{\mu_1 - L_2 \,,\, L_1 - \mu_2}$.

\medskip 

\noindent \textbf{Difference-of-Convex algorithm (DCA) iteration.} 
\begin{align}\label{eq:DCA_it}\tag{$\text{DCA}^{}$}
\begin{aligned}
    \text{1. } & \text{Select } g_2 \in \partial f_2(x). \\
    \text{2. } & \text{Select } x^{+} \in \argmin_{w\in \mathbb{R^d}} \{f_1(w) - \langle g_2, w \rangle\}.
\end{aligned}    
\end{align}%

It is easy to see that the optimality condition in the definition of $x^+$ implies the existence of $g_1^+ \in \partial f_1(x^+)$ such that $g_1^+ = g_2$. This is the only characterization of $x^+$ that will be used in our derivations for the smooth case (S).

An iteration of DCA is not well-defined when the minimum in the definition of $x^+$ does not exist. One can guarantee that all iterations are well-defined (although not necessarily uniquely defined)
when for example $f_1$ is supercoercive, meaning $\lim_{\|w\|\rightarrow \infty} \frac{f_1(w)}{\|w\|} = \infty$ \cite[Definition 11.11]{Bauschke_Cvx_Analysis_2011}. A sufficient condition for supercoercivity is for $f_1$ to be strongly convex \cite[Corollary 11.17]{Bauschke_Cvx_Analysis_2011}, i.e., $\mu_1 > 0$.

Recall that a critical point $x^*$ is such that $\partial f_2(x^*) \cap \partial f_1(x^*) \neq \emptyset$. When both functions $f_1$ and $f_2$ are smooth, any critical point $x^*$ is clearly stationary as we have $\nabla F(x^*) = \nabla f_1(x^*) - \nabla f_2(x^*) = 0$. When only $f_2$ is smooth, we have $\partial F(x^*) = \partial f_1(x^*) - \nabla f_2(x^*)$ (see \cite[Exercise 10.10]{RockWets98}) and criticality also implies stationarity, since $0 \in \partial F(x^*)$. However, when only $f_1$ is smooth, we can only guarantee the inclusion $\partial (-f_2) (x^*) \subseteq -\partial f_2(x^*)$
\cite[Corollary 9.21]{RockWets98}, implying only $\partial F(x^*) \subseteq \nabla f_1(x^*) - \partial f_2(x^*)$, and critical points may not be stationary. Finally, when both functions are nonsmooth, a necessary condition for stationarity of $x^*$ in \eqref{eq:DCA_it} is $\partial f_2(x^*) \subseteq \partial f_1(x^*)$ (see \cite[Proposition 3.1]{Hiriart_Urruty1989_nec_suff_conds_global_optimality}), strictly stronger that criticality. 

\smallskip
\begin{proposition}[Sufficient condition for decrease] \label{prop:sufficient_decrease}
Let $f_1 \in \mathcal{F}_{\mu_1,L_1}$ and $f_2 \in \mathcal{F}_{\mu_2,L_2}$. If $\mu_1 + \mu_2 \ge 0$, then the objective function $F = f_1 - f_2$ decreases after each iteration of \eqref{eq:DCA_it}. Moreover, if $\mu_1 + \mu_2 > 0$ that objective decrease is strict, unless $x^+ = x$.
\end{proposition}%
\begin{proof} Our proof is inspired from \cite[Theorem 3, Proposition 2]{Dinh_Thi_dca_1997}. Using \cref{lemma:hypoconvex_smooth_quad_bounds} on function $f_1$ with $(x,y)=(x,x^+)$ and on function $f_2$ with $(x,y)=(x^+,x)$, and summing the resulting inequalities, we obtain, where $g_1^+ \in \partial f_1(x^+)$ and $g_2 \in \partial f_2(x)$:%
$$\begin{aligned}
& \frac{\mu_1 + \mu_2}{2} \|x^+ - x\|^2
        {}\leq{}
        f_1(x) - f_1(x^+) - \langle g_1^+\,,\, x - x^+ \rangle 
        {}+{} 
        \\ & {}\quad{} f_2(x^+) - f_2(x) - \langle g_2\,,\, x^+ - x \rangle
        \leq        
        \frac{L_1 + L_2}{2} \|x^+ - x\|^2.
\end{aligned}$$%
By using $F(x) = f_1(x) - f_2(x)$ and $g_1^+ = g_2$ we get:%
$$
        \frac{\mu_1 + \mu_2}{2} \|x - x^+\|^2
            \leq
        F(x) - F(x^+) 
        \leq
        \frac{L_1 + L_2}{2} \|x - x^+\|^2,
$$
which is enough to prove both parts of the proposition.%
\end{proof}%
\medskip \textbf{Notation:} Superscripts  indicate the iteration index (e.g., $x^k$ represents the $k$-th iterate). We use $G(x)$ or $G$ to denote a difference of subgradients $g_1 - g_2$ for some $g_1 \in \partial f_1(x)$ and $g_2 \in \partial f_2(x)$.

\section{Rates for the smooth case (S)}
In this section we study the case (S) where at least one function $f_1$ or $f_2$ is smooth.%
\begin{theorem}[One-step decrease]\label{thm:one_step_decrease_dca}
Let $f_1 \in \mathcal{F}_{\mu_1,L_1}$ and $f_2 \in \mathcal{F}_{\mu_2,L_2}$ satisfying \cref{{assumption:curvatures_on_F}},  assume at least $f_1$ or $f_2$ is smooth, and assume $\mu_1 + \mu_2 > 0$ or $\mu_1=\mu_2=0$. Then after one step of \eqref{eq:DCA_it} we have
\begin{align}\label{eq:tight_dist_1_with_sigmas}
    F(x)-F(x^+)
    {}\geq{} \sigma_i \tfrac{1}{2} \|g_1 - g_2\|^2 + \sigma_i^+ \tfrac{1}{2} \|g_1^+ - g_2^+\|^2, 
\end{align}%
with $g_1 \in \partial f_1(x)$, $g_1^+ \in \partial f_1(x^+)$, $g_2 \in \partial f_2(x)$, $g_2^+ \in \partial f_2(x^+)$ and the expressions of $\sigma_i, \sigma_i^+ \geq 0$ correspond to one of the eight regimes indexed by $i = 1,\dots,8$ provided in \cref{tab:DCA_regimes_one_step} according to the values of parameters $L_1, L_2, \mu_1, \mu_2$. %
\end{theorem}%

The eight regimes appearing in \cref{tab:DCA_regimes_one_step} are depicted on \cref{fig:all_together}. It is interesting to observe the striking symmetry between the regimes with odd and even indices. More specifically, formulas in \cref{tab:DCA_regimes_one_step} for $p_{2,4,6,8}$ derive from $p_{1,3,5,7}$ by exchanging $L_1 \leftrightarrow L_2$, $\mu_1 \leftrightarrow \mu_2$, and $\sigma_i \leftrightarrow \sigma_i^+$. The proof of \cref{thm:one_step_decrease_dca} is postponed to \cref{sec:proof_one_step_analysis}.%
\begin{conjecture}\label{conjecture:tightness_one_N=1}
All eight regimes outlined in \cref{thm:dca_rates_N_steps} are tight, i.e., the corresponding lower bounds on the objective decrease cannot be improved.
\end{conjecture}%
\begin{table*}[t]
\centering
\caption{Exact decrease after one step: $F(x)-F(x^+)\geq \sigma_i \frac{1}{2} \|g_1-g_2\|^2 + \sigma_i^+ \frac{1}{2} \|g_1^+ - g_2^+ \|^2$ 
(see \cref{thm:one_step_decrease_dca}), where $\sigma_i, \sigma_i^+ \geq 0$ and $p_i = \sigma_i + \sigma_i^+$, with $i=1,\dots,8$. Scalar $\alpha_i \geq 0$ is a parameter of the proofs.}
\label{tab:DCA_regimes_one_step}
\resizebox{\textwidth}{!}{%
\begin{tabular}{|c|c|c|c|c|c|}
\hline
\textbf{Regime} &
  \centered{\textbf{$\sigma_i$}} &
  \centered{\textbf{$\sigma_i^+$}} &
  \textbf{$\alpha_i$} &
  \textbf{Domain} \\ [3pt] 
  \hline
\textbf{$p_1$} &
  $L_2^{-1} \dfrac{L_2-\mu_1}{L_1-\mu_1}$ &
  $L_2^{-1} \bigg( 1 + \dfrac{L_2^{-1}-L_1^{-1}}{\mu_1^{-1} - L_1^{-1} } \bigg) $ &
  $\dfrac{\mu_1}{L_2} \dfrac{L_1-L_2}{L_1-\mu_1}$ 
  & \makecell{ $L_1 \geq L_2 > \mu_1 \geq 0$; 
  $\mu_2 \geq 0$ or
  \\ 
  $\mu_1 > -\mu_2 > 0$ and $\mu_1^{-1} + \mu_2^{-1} + L_2^{-1}{}\leq{} L_1^{-1} \big(2 + \frac{L_2}{\mu_2} \big)$
  }
  \\ [8pt] \hline
\textbf{$p_2$} &
  $L_1^{-1} \bigg( 1 + \dfrac{L_1^{-1}-L_2^{-1}}{\mu_2^{-1} - L_2^{-1}} \bigg) $ &
  $L_1^{-1} \dfrac{L_1-\mu_2}{L_2-\mu_2}$ &
  $\dfrac{\mu_2}{L_1} \dfrac{L_2-L_1}{L_2-\mu_2}$
  & \makecell{ $L_2 \geq L_1 > \mu_2 \geq 0$; 
  $\mu_1 \geq 0$ or
  \\
  $\mu_2 > -\mu_1 > 0$ and $\mu_1^{-1} + \mu_2^{-1} + L_1^{-1} {}\leq{} L_2^{-1} \big(2 + \frac{L_1}{\mu_1} \big)$
  }
  \\ [8pt] \hline
\textbf{$p_3$} &
  $\dfrac{L_1^{-1}  \big( \mu_1^{-1} + \mu_2^{-1} + L_2^{-1} \big) }{ \mu_1^{-1} + \mu_2^{-1} + L_2^{-1} - L_1^{-1} }$ &
  $\dfrac{1}{L_2+\mu_2}$ &
  $\dfrac{-\mu_2}{L_2+\mu_2}$ 
  & \makecell{ $\mu_1 > -\mu_2 > 0$; $L_2 > \mu_1$; $L_1 > \mu_2$;
  \\
  $L_1^{-1} \big(2 + \frac{L_2}{\mu_2} \big) {}\leq{} \mu_1^{-1} + \mu_2^{-1} + L_2^{-1} {}\leq{} 0$
  }
  \\ [8pt] \hline
\textbf{$p_4$} &
  $\dfrac{1}{L_1+\mu_1}$ &
  $\dfrac{L_2^{-1} \big(\mu_1^{-1} + \mu_2^{-1} + L_1^{-1}\big)}{\mu_1^{-1} + \mu_2^{-1} + L_1^{-1} - L_2^{-1}} $ &
  $\dfrac{-\mu_1}{L_1+\mu_1}$ 
  & \makecell{ $\mu_2 > -\mu_1 > 0$; $L_2 > \mu_1$; $L_1 > \mu_2$;
  \\
  $L_2^{-1} \big(2 + \frac{L_1}{\mu_1} \big) {}\leq{} \mu_1^{-1} + \mu_2^{-1} + L_1^{-1} {}\leq{} 0$
  }
  \\ [8pt] \hline
\textbf{$p_5$} &
  $0$ 
  &
  $\dfrac{\mu_1+\mu_2}{\mu_2^2}$
  &
  $\dfrac{\mu_1+\mu_2}{-\mu_2}$
  & \makecell{$\mu_1 > -\mu_2 > 0$; $L_1 > \mu_2$; \\
    $
    \max 
    \big\{
    L_1^{-1} \big( 2 + \tfrac{L_2}{\mu_2} \big)
    \, , \,
    0
    \big\}
    {}<{}
    \mu_1^{-1} + \mu_2^{-1} + L_2^{-1}$}
  \\ [8pt] \hline
\textbf{$p_6$} &
  $\dfrac{\mu_1+\mu_2}{\mu_1^2}$
  &
  $0$ &
  $\dfrac{\mu_1+\mu_2}{-\mu_1}$
  & \makecell{$\mu_2 > -\mu_1 > 0$; $L_2> \mu_1$; \\
    $\,
    \max 
    \big\{
    L_2^{-1} \big(2 + \tfrac{L_1}{\mu_1}\big)
    \, , \,
    0
    \big\}
    {}<{}
    \mu_1^{-1} + \mu_2^{-1} + L_1^{-1} $} \\ [8pt] \hline
\textbf{$p_7$} &
  $0$
  &
  $\dfrac{L_2 + \mu_1}{L_2^2}$ &
  $\dfrac{\mu_1}{L_2}$ 
  & \makecell{$L_1 > \mu_1 > L_2 > 0$; \\
    $\mu_2 \big( \mu_1^{-1} + \mu_2^{-1} + L_2^{-1} \big) {} \geq {} 0 $ }
  \\ [8pt] \hline
\textbf{$p_8$} &
  $\dfrac{L_1 + \mu_2}{L_1^2}$
  &
  $0$ &
  $\dfrac{\mu_2}{L_1}$ 
  & \makecell{$L_2 > \mu_2 > L_1 > 0$; \\
    $\mu_1 \big( \mu_1^{-1} + \mu_2^{-1} + L_1^{-1} \big) {} \geq {} 0 $ } \\ [8pt] \hline%
\end{tabular}}
\end{table*}%
\begin{figure*}[t]
    \centering
    \includegraphics[width=.49\textwidth]{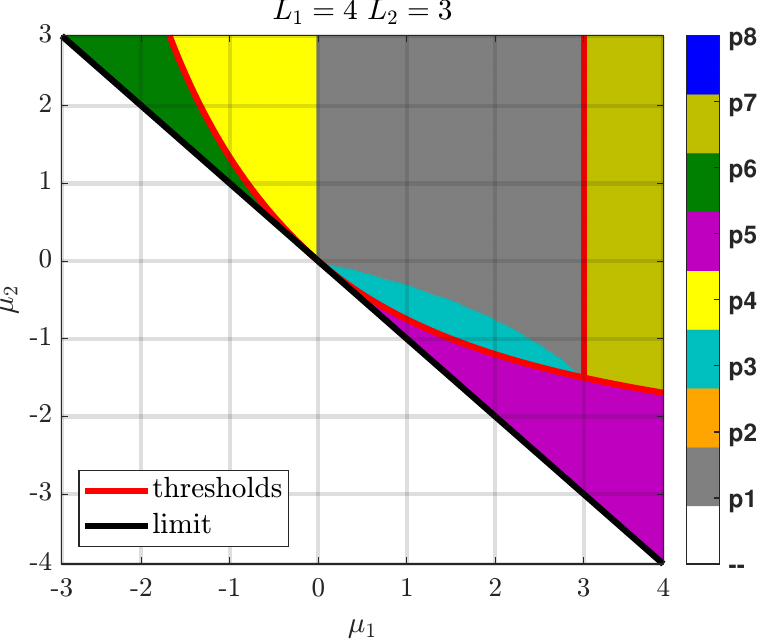}    
        {}{}
    \includegraphics[width=.49\textwidth]{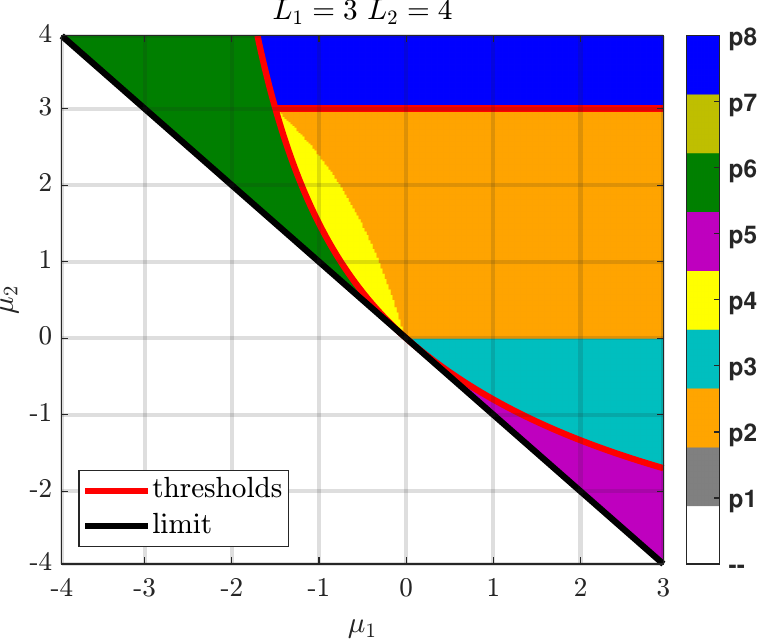}
    \caption{All regimes after one \eqref{eq:DCA_it} iteration (see \cref{tab:DCA_regimes_one_step} for precise expressions). The sufficient decrease condition from \cref{prop:sufficient_decrease} holds above the \textit{limit} line $\mu_1 + \mu_2 = 0$. Regimes $p_1$, $p_2$, $p_3$ and $p_4$, lying inside the area delimited by \textit{thresholds} (defined in \eqref{eq:threshold_cond_p3} for $\mu_1 > 0$ and \eqref{eq:threshold_cond_p4} for $\mu_2 > 0$), are (numerically conjectured to be) tight after $N$ steps.
    }
    \label{fig:all_together}
\end{figure*}%
\medskip \begin{theorem}[DCA sublinear rates]\label{thm:dca_rates_N_steps}
Let $f_1 \in \mathcal{F}_{\mu_1,L_1}$ and $f_2 \in \mathcal{F}_{\mu_2,L_2}$ satisfying \cref{{assumption:curvatures_on_F}},  assume at least $f_1$ or $f_2$ is smooth, and assume $\mu_1 + \mu_2 > 0$ or $\mu_1=\mu_2=0$. Then after $N$ iterations of \eqref{eq:DCA_it} starting from $x^0$ we have
\begin{align}\label{eq:rate_N_steps_no_F*}
    \tfrac{1}{2} \min_{0\leq k \leq N} \{\|g_1^k - g_2^k\|^2\} {}\leq{}
    \frac{F(x^0)-F(x^N)}{p_i(L_1, L_2, \mu_1, \mu_2) N},
\end{align}
where $g_1^k \in \partial f_1(x^k)$ and $g_2^k \in \partial f_2(x^k)$ for all $k=0,\dots,N$ and $p_i=\sigma_i + \sigma_i^+$ is given in \cref{tab:DCA_regimes_one_step}. Additionally, if $F$ is nonconcave (i.e., $L_1 > \mu_2$), the following  also holds
\begin{align}\label{eq:rate_N_steps_with_F*}
    \hspace{-.78em}
    \tfrac{1}{2} \min_{0\leq k \leq N} \{\|g_1^k - g_2^k\|^2\} \leq
    \frac{F(x^0)-F^*}{p_i(L_1, L_2, \mu_1, \mu_2) N + \frac{1}{L_1 - \mu_2}}.
\end{align}
\end{theorem}%

Regimes $p_1$ and $p_2$ correspond in part to the standard setting of \eqref{eq:DCA_it} where both functions are convex ($\mu_1 \ge 0$, $\mu_2 \ge 0$). Whether we are in $p_1$ or $p_2$ depends on which is larger among $L_1$ and $L_2$. These two regimes also require the objective $F \in \mathcal{F}_{\mu_1-L_2,\ L_1-\mu_2}$ to be nonconvex ($L_2>\mu_1$) and nonconcave ($L_1 > \mu_2$). They were first established in the standard difference-of-convex case by \cite{abbaszadehpeivasti2021_DCA}, using performance estimation. All the  other regimes we describe are novel. 

\cref{thm:one_step_decrease_dca} actually extends regimes $p_1$ and $p_2$ beyond the difference-of-convex case, i.e., to situations where one of the functions is hypoconvex, up to a certain threshold. Regime $p_1$ applies in a subdomain where $\mu_1 > 0 > \mu_2$, while $p_2$ holds in the complementary subdomain with $\mu_2 > 0 > \mu_1$. Beyond this threshold, regimes $p_3$ and $p_4$ emerge, under the conditions:%
\begin{align}\label{eq:threshold_cond_p3} \tag{$S_1$}
        &S_1 {}:={} \mu_1^{-1} + \mu_2^{-1} + L_2^{-1} \leq 0
        \text{\,\, if } \mu_1 > 0;
        \\ & S_2 {}:={} \mu_1^{-1} + \mu_2^{-1} + L_1^{-1} \leq 0 
        \text{\,\, if } \mu_2 > 0 \label{eq:threshold_cond_p4} \tag{$S_2$},
\end{align}%
with $p_3$ confined by \eqref{eq:threshold_cond_p3} and $p_4$ by \eqref{eq:threshold_cond_p4}. 

\smallskip It turns out that, in conjunction with the conditions $\mu_1, \mu_2 \leq \min \{L_1, L_2\}$, i.e., $F$ is nonconvex and nonconcave, thresholds conditions $S_1 = 0$ and $S_2 = 0$ (depicted by the red curves from \cref{fig:all_together}) actually distinguish between tight and non-tight regimes for DCA rates.

\begin{conjecture}\label{conjecture:tightness_p1_p2_p3_p4}
The DCA rates corresponding to regimes $p_1$, $p_2$, $p_3$ and $p_4$ from \cref{thm:dca_rates_N_steps} are tight for any number of iterations $N$.%
\end{conjecture}%
Further investigation is needed to determine  exact convergence rates for regimes $p_5$ and $p_6$, located within the region between the threshold line and the maximum range of $\mu_1 + \mu_2$ on \cref{fig:all_together}. Indeed, our numerical investigations indicate that the  sublinear rates for $p_5$ and $p_6$ in \cref{thm:dca_rates_N_steps} are not tight beyond a single iteration. However, we are able to conjecture the following asymptotic behavior.

\begin{conjecture}\label{conjecture:Regimes_p5_p6_above_thr}
Assume $L_1 > \mu_2$ and $L_2 > \mu_1$. Then for a sufficiently large number of iterations $N$, the exact sublinear rates corresponding to regimes $p_5$ and $p_6$ with domains defined in \cref{tab:DCA_regimes_one_step} correspond to
$$
    \tfrac{1}{2} \min_{0\leq k \leq N} \{\|g_1^k - g_2^k\|^2\} {}\leq{}
    \frac{F(x^0)-F(x^N)}{a(L_1, L_2, \mu_1, \mu_2) N + b},
$$
where $a \in \{p_5^{\infty}, p_6^{\infty} \}$, with $p_5^{\infty} := \frac{(L_2+\mu_1)(\mu_1+\mu_2)}{(L_2+\mu_2) \mu_1^2}$ and
$p_6^{\infty} := \frac{(L_1+\mu_2)(\mu_1+\mu_2)}{(L_1+\mu_1) \mu_2^2}$, and $b$ is an unidentified term independent of $N$.
\end{conjecture}

The final two regimes $p_7$ and $p_8$ are included to ensure a comprehensive analysis. Regime $p_7$ belongs to the scenario where $\mu_1 \geq L_2$, namely the objective function $F$ is (strongly) convex. Likewise, $p_8$ arises when $\mu_2 > L_1$, indicating a (strongly) concave objective (and unbounded from below). In these two cases, the rates are actually linear, their precise description being left for future work. %

\smallskip
\begin{remark}
In the specific convex scenario $\mu_1 = \mu_2 = 0$, both regimes $p_1$ and $p_2$ apply, as outlined in \cite[Corollary 3.1]{abbaszadehpeivasti2021_DCA}. For completeness, the one-step decrease in this case is given by:
$F(x)-F(x^+)\geq \frac{1}{2L_1} \|g_1 - g_2\|^2 + \frac{1}{2L_2} \|g_1^+ - g_2^+\|^2$.
\end{remark}%

\medskip  All the above results hold when at least one of the functions $f_1$ and $f_2$ is smooth. When exactly one of them is smooth, i.e., when the other is nonsmooth, some expressions in \cref{tab:DCA_regimes_one_step} become simpler, and we give an explicit description below.

\begin{corollary}[Rates with one nonsmooth term]
\label{corollary:dca_rates_nonsmooth_N_steps}
Let $f_1 \in \mathcal{F}_{\mu_1,L_1}$ and $f_2 \in \mathcal{F}_{\mu_2,L_2}$, where exactly one function $f_1$ or $f_2$ is smooth, and assume $\mu_1 + \mu_2 > 0$ or $\mu_1=\mu_2=0$. Consider $N$ iterations of \eqref{eq:DCA_it} starting from $x^0$. Then:%
$$
    \tfrac{1}{2} \min_{0\leq k \leq N} \{\|g_1^k - g_2^k\|^2\} {}\leq{}
    \frac{F(x^0)-F(x^N)}{p_i(L_1, L_2, \mu_1, \mu_2) N},
$$%
where $g_1^k \in \partial f_1(x^k)$ and $g_2^k \in \partial f_2(x^k)$ for all $k=0,\dots,N$ and $p_i$ is provided in \cref{tab:DCA_regimes_nonsmooth}. Additionally, if $F$ is nonconcave (i.e., $L_1 > \mu_2$):%
$$
    \tfrac{1}{2} \min_{0\leq k \leq N} \{\|g_1^k - g_2^k\|^2\} {}\leq{}
    \frac{F(x^0)-F^*}{p_i(L_1, L_2, \mu_1, \mu_2) N + \frac{1}{L_1 - \mu_2}}.
$$%
\end{corollary}%
\begin{table}[t]
\centering
\caption{The exact decrease after one iteration with $f_1$ or $f_2$ nonsmooth: $F(x)-F(x^+)\geq \sigma_i \frac{1}{2} \|g_1-g_2\|^2 + \sigma_i^+ \frac{1}{2} \|g_1^+ - g_2^+ \|^2$
, where $\sigma_i, \sigma_i^+ \geq 0$ and $p_i = \sigma_i + \sigma_i^+$. 
}
\label{tab:DCA_regimes_nonsmooth}
\resizebox{\linewidth}{!}{%
\begin{tabular}{|@{}c@{}|@{}c@{}|@{}c@{}|@{}c@{}|@{}c@{}|@{}c@{}|}
\hline
\textbf{Regime} &
  \textbf{$\sigma_i$} &
  \textbf{$\sigma_i^+$} &
  \textbf{$\alpha_i$} &
  \textbf{Domain} \\ [2pt] 
  \hline
\textbf{$p_{1,7}$} &
  $0$ &
  $\frac{L_2 + \mu_1}{L_2^2} $ &
  $\frac{\mu_1}{L_2}$ 
  & \makecell{ $L_1 = \infty > L_2 > \mu_1 \geq 0$ 
  \\ 
  $\mu_2 \big(\frac{1}{\mu_1} + \frac{1}{\mu_2} + \frac{1}{L_2}\big) \geq 0$
  }
  \\  [6pt] \hline
\textbf{$p_{2,8}$} &
  $\frac{L_1 + \mu_2}{L_1^2}$ &
  $0$ &
  $\frac{\mu_2}{L_1}$
  & \makecell{ $L_2 = \infty > L_1 > \mu_2 \geq 0$ 
  \\
  $\mu_1 \big(\frac{1}{\mu_1} + \frac{1}{\mu_2} + \frac{1}{L_1}\big) \geq 0$
  }
  \\  [6pt] \hline
\textbf{$p_3$} &
  $\frac{\frac{1}{L_1} \big( \frac{1}{\mu_1} + \frac{1}{\mu_2} \big) }{ \frac{1}{\mu_1} + \frac{1}{\mu_2}- \frac{1}{L_1} }$ 
  &
  $0$ &
  $0$
  & \makecell{ $L_2 = \infty$; $\mu_1 > -\mu_2 > 0$
  }
  \\ [8pt] \hline
\textbf{$p_4$} &
  $0$ &
  $\frac{\frac{1}{L_2} \big( \frac{1}{\mu_1} + \frac{1}{\mu_2} \big) }{ \frac{1}{\mu_1} + \frac{1}{\mu_2}- \frac{1}{L_2} }$ 
  &
  $0$ 
  & \makecell{ $L_1 = \infty$; $\mu_2 > -\mu_1 > 0$
  }
  \\ [8pt] \hline
\textbf{$p_5$} &
  $0$ 
  &
  $\frac{\mu_1+\mu_2}{\mu_2^2}$
  &
  $\frac{\mu_1+\mu_2}{-\mu_2}$
  & \makecell{$L_1=\infty$; $\mu_1 > -\mu_2 > 0$ \\
    $0{}<{} \mu_1^{-1} + \mu_2^{-1} + L_2^{-1}$}
  \\  [6pt] \hline
\textbf{$p_6$} &
  $\frac{\mu_1+\mu_2}{\mu_1^2}$
  &
  $0$ &
  $\frac{\mu_1+\mu_2}{-\mu_1}$
  & \makecell{$L_2=\infty$; $\mu_2 > -\mu_1 > 0$ \\
    $ 0 {}<{}  \mu_1^{-1} + \mu_2^{-1} + L_1^{-1} $} \\ [6pt] \hline%
\end{tabular}}%
\end{table}%
\cref{corollary:dca_rates_nonsmooth_N_steps} is derived by setting $L_1=\infty$ or $L_2=\infty$ in \cref{thm:dca_rates_N_steps} and in the corresponding entries from \cref{tab:DCA_regimes_one_step}. It shows identical rates as in \cite[Corollary 3.1]{abbaszadehpeivasti2021_DCA} for $\mu_1, \mu_2 \geq 0$, while extending them to scenarios involving hypoconvex $f_1$ or $f_2$, up to the thresholds \eqref{eq:threshold_cond_p3} and \eqref{eq:threshold_cond_p4}. Beyond these thresholds, regimes $p_5$ and $p_6$ emerge. Notably, we observe in \cref{tab:DCA_regimes_nonsmooth} that regimes $p_7$ and $p_8$ condense to regimes $p_1$ and $p_2$, respectively. When $L_1=\infty$, only regimes $p_{1,4,5}$ hold, $p_4$ covering the entire domain with $\mu_1 < 0$, and threshold \eqref{eq:threshold_cond_p3} separating $p_1$ and $p_5$ for $\mu_1 \geq 0$; similarly for $L_2=\infty$, $p_3$ covering the domain with $\mu_2 < 0$ and threshold \eqref{eq:threshold_cond_p4} separating $p_2$ and $p_6$ when $\mu_2 \geq 0$.
%
\section{Proofs}\label{sec:proofs}
\subsection{Performance estimation problem (PEP) and interpolation }\label{sec:Performance_Estimation}
The performance estimation problem (PEP), introduced by Drori and Teboulle in their seminal paper \cite{drori_performance_2014} and later refined by Taylor et al. in \cite{taylor_smooth_2017}, serves as a framework for analyzing tight convergence rates in various optimization methods. It casts the search for the worst-case function as an optimization problem, restricting the analysis solely to the iterations by employing certain inequalities which are necessary and sufficient to interpolate functions from the class of interest, such as \cref{thm:interp_hypo_characterization_min} outlined below. Subsequently, the problem is reformulated into a convex semidefinite program, which can be solved numerically. For an in-depth guide on setting the PEP for DCA, we direct readers to \cite[Section 3]{abbaszadehpeivasti2021_DCA}. 

\smallskip \begin{theorem}[$\mathcal{F}_{\mu,L}$-interpolation] \label{thm:interp_hypo_characterization_min}
Given an index set $\mathcal{I}$, let $\mathcal{T}:=\{(x^i, g^i, f^i)\}_{i \in \mathcal{I}} \subseteq \mathbb{R}^d \times \mathbb{R}^d \times \mathbb{R}$ be a set of triplets. There exists $f \in \mathcal{F}_{\mu,L}$, with $L>0$ and $\mu \le L$, satisfying $f(x^i)=f^i$ and $g^i \in \partial f(x^i)$ for all $i \in \mathcal I$, if and only if the following interpolation inequality holds for every pair of indices $(i,j)$, with $i,j \in \mathcal{I}$%
\begin{align}\label{eq:Interp_hypoconvex}\tag{$I_{i,j}$}
\begin{aligned}
    f^{i}-f^{j}-&\langle g^{j}, x^{i}-x^{j} \rangle 
    \geq
    \frac{1}{2L} \|g^{i}-g^{j}\|^{2} +
    \\ & +
    \frac{\mu}{2L(L-\mu)} \|g^{i}-g^{j} - L (x^i-x^j)\|^2.    
\end{aligned}
\end{align}%
\end{theorem}
\smallskip \cref{thm:interp_hypo_characterization_min} is introduced in \cite[Theorem 4]{taylor_smooth_2017} for $\mu \geq 0$ and extended to negative lower curvatures in \cite[Theorem 3.1]{rotaru2022tight}.

\medskip
We present below in \eqref{PEP:DCA} the general performance estimation setup for DCA which can be directly integrated in one of the specialized software packages PESTO \cite{PESTO} (for Matlab) or PEPit \cite{PESTO_python} (for Python):%
\begin{align}\label{PEP:DCA}\tag{PEP-DCA}
    \begin{aligned}
        \maximize \,\, & 
        \frac{ \frac{1}{2} \min \limits_{0 \leq k \leq N} \{\|g_1^k - g_2^k \|^2 \}} {F(x^0)-F(x^N)}
        \\
        \text { subject to }
            \,\, & 
            \{(x^k,g_{1,2}^k,f_{1,2}^k)\}_{k \in \mathcal{I}} \text{ satisfy \eqref{eq:Interp_hypoconvex}} \\
            \,\, & g_1^{k+1} = g_2^{k} \quad k \in\{0, \ldots, N-1\}.
    \end{aligned}
\end{align}%
The decision variables are $x^k$, $g_1^k$, $g_2^k$, $f_1^k$, $f_2^k$, with $k \in \mathcal{I}=\{0,\dots,N\}$. The numerical solutions of \eqref{PEP:DCA} aided our derivations, enabling verification of the rates from \cref{thm:dca_rates_N_steps}, conjecturing their tightness, and selecting which interpolation inequalities are needed in the proofs among the complete set from \cref{thm:interp_hypo_characterization_min}. While employing PEP can sometimes make proofs difficult to comprehend, we provide in \cref{sec:proof_one_step_analysis} a more accessible demonstration of the various regimes. 

\medskip 
\cref{conjecture:tightness_one_N=1}, \cref{conjecture:tightness_p1_p2_p3_p4} and \cref{conjecture:Regimes_p5_p6_above_thr} are supported by extensive numerical evidence obtained through the solution of a large number of performance estimation problems \eqref{PEP:DCA}, spanning the range of allowed parameters $L_1, L_2, \mu_1, \mu_2$. 

\subsection{Proof of \cref{thm:one_step_decrease_dca}}\label{sec:proof_one_step_analysis}
\noindent \cref{thm:interp_hypo_characterization_min} is the key tool to prove \cref{thm:one_step_decrease_dca}.%
\begin{proof}
    Let $g_1 \in \partial f_1(x)$, $g_1^+ \in \partial f_1(x^+)$,  $g_2 \in \partial f_2(x)$ and $g_2^+ \in \partial f_2(x^+)$, with $g_1^+ = g_2$. We use the notation: $\Delta x := x - x^{+}$, $G := g_1 - g_2$, $G^+ := g_1^+ - g_2^+$ and $\Delta F(x) := F(x) - F(x^+)$. By writing \eqref{eq:Interp_hypoconvex} for function $f_1$ with the iterates $(x,x^+)$ we obtain: 
    \begin{align}\label{eq:f_1_evals_in_x_x+}
\begin{aligned}
    f_1(x)-f_1(x^{+}) {}-{} &\langle g_1^+, \Delta x \rangle 
    \geq
    \frac{1}{2L_1} \|G\|^{2} +
    \\ & {}+{}
    \frac{\mu_1}{2L_1(L_1-\mu_1)} \|G - L_1 \Delta x\|^2
\end{aligned}
\end{align}%
    and for function $f_2$ with the iterates $(x^+,x)$ we get:
    \begin{align}\label{eq:f_2_evals_in_x+_x}
\begin{aligned}
    f_2(x^+)-f_2(x) {}+{} &\langle g_2, \Delta x \rangle 
    \geq
    \frac{1}{2L_2} \|G^+\|^{2} +
    \\ & {}+{}
    \frac{\mu_2}{2L_2(L_2-\mu_2)} \|G^+ - L_2 \Delta x\|^2.
\end{aligned}
\end{align}%
    Summing them up and performing simplifications we get:
    \begin{align}\label{eq:consec_funs_it}
    \begin{aligned}
    \Delta F(x)
    {}\geq{}
    \frac{1}{2L_1} \|G\|^{2} + 
    \frac{\mu_1}{2L_1(L_1-\mu_1)} \|G - L_1 \Delta x\|^2
    \\ {}+{}
    \frac{1}{2L_2} \|G^+\|^{2} +
    \frac{\mu_2}{2L_2(L_2-\mu_2)} \|G^+ - L_2 \Delta x\|^2.    
    \end{aligned}   
    \end{align}
By writing \eqref{eq:Interp_hypoconvex} for function $f_2$ with the iterates $(x,x^+)$:%
\begin{align}
\begin{aligned}
    f_2(x)-f_2(x^+){}-{}&\langle g_2^{+}, \Delta x\rangle 
    \geq
    \frac{1}{2L_2} \|G^{+}\|^{2} +
    \\ & {}+{}
    \frac{\mu_2}{2L_2(L_2-\mu_2)} \|G^{+} - L_2 \Delta x\|^2
\end{aligned}
\end{align}%
and summing it up with \eqref{eq:f_2_evals_in_x+_x} we get:%
\begin{align}\label{eq:dist_1_correction}
\hspace{-.9em}
    \langle G^{+}, \Delta x \rangle 
        \geq
    \frac{1}{L_2} \|G^+\|^{2} +
    \frac{\mu_2}{L_2(L_2-\mu_2)} \|G^+ - L_2 \Delta x\|^2.
\end{align}%
Similarly, by writing \eqref{eq:Interp_hypoconvex} for function $f_1$ with the iterates $(x^+,x)$ and summing it up with \eqref{eq:f_1_evals_in_x_x+} we get:%
\begin{align}\label{eq:dist_1_correction_p_even}
\begin{aligned}
     \langle G, \Delta x  \rangle 
    \geq
    \frac{1}{L_1} \|G\|^{2} +
    \frac{\mu_1}{L_1(L_1-\mu_1)} \|G - L_1 \Delta x\|^2.
\end{aligned}
\end{align}%
The proofs only involve adjusting the right-hand side of \eqref{eq:consec_funs_it} using either inequality \eqref{eq:dist_1_correction} or inequality \eqref{eq:dist_1_correction_p_even}, weighted by scalars $\alpha > 0$, which depend on the curvatures (see \cref{tab:DCA_regimes_one_step}). Specifically, to establish regimes $p_{1,3,5,7}$ we substitute $\alpha$ in:
\begin{align}\label{eq:ineq_param_in_alpha}
    \begin{aligned}
    \Delta F(x)
    {}\geq{}
    \frac{1}{2L_1} \|G\|^{2} + 
    \frac{\mu_1}{2L_1(L_1-\mu_1)} \|G - L_1 \Delta x\|^2 +
    \\
    \frac{1+2\alpha}{2L_2} \|G^+\|^{2} +
    \frac{\mu_2(1+2\alpha)}{2L_2(L_2-\mu_2)} \|G^+ - L_2 \Delta x\|^2 + 
    \\ -\alpha \langle G^{+}, \Delta x  \rangle
    \end{aligned}   
\end{align}%
and to demonstrate regimes $p_{2,4,6,8}$ we plug in $\alpha$ in:%
$$
    \begin{aligned}
    \Delta F(x)
    {}\geq{}
    \frac{1}{2L_2} \|G^+\|^{2} +
    \frac{\mu_2}{2L_2(L_2-\mu_2)} \|G^+ - L_2 \Delta x\|^2     
    \\ 
    \frac{1+2\alpha}{2L_1} \|G\|^{2} + 
    \frac{\mu_1(1+2\alpha)}{2L_1(L_1-\mu_1)} \|G - L_1 \Delta x\|^2 +
    \\ 
    -\alpha \langle G, \Delta x \rangle.    
    \end{aligned} 
$$%

Since the proof relies purely on algebraic manipulation, by exploiting the symmetry in the right-hand side of the two inequalities, we only demonstrate the regimes $p_1, p_3, p_5, p_7$, while the complementary ones ($p_2, p_4, p_6, p_8$) are obtained by interchanging: (i) curvature indices $1$ and $2$; and (ii) $G$ and $G^+$; specifically, $p_1 \leftrightarrow p_2$, $p_3 \leftrightarrow p_4$, $p_5 \leftrightarrow p_6$ and $p_7 \leftrightarrow p_8$ (also see \cref{tab:DCA_regimes_one_step}).

\smallskip \textbf{Regime $p_1$:} It corresponds to $L_1 \geq L_2 \geq \mu_1 \geq 0$; in particular, $L_1 \geq \max \{L_2,\mu_1,\mu_2\}$. By setting $\alpha = \frac{\mu_1}{L_2} \frac{L_1 - L_2}{L_1 - \mu_1}$ in \eqref{eq:ineq_param_in_alpha}, after simplifications and building the squares we get:%
$$
\begin{aligned}
\Delta F(x)
 \geq
    \frac{L_2-\mu_1}{L_2(L_1-\mu_1)}
    \frac{\|G\|^{2}}{2}
    + \frac{\mu_1}{L_2(L_1-\mu_1)} \frac{\|G - L_2 \Delta x\|^2}{2}
    \\ \,\, 
    + \frac{1}{L_2} \bigg[1 + \frac{\mu_1(L_1-L_2)}{L_2(L_1-\mu_1)}\bigg]
     \frac{\|G^+\|^2}{2}
    \\ \,\, 
    + \frac{\mu_1 \frac{L_1}{L_2} \, \mu_2 \big[\frac{1}{\mu_1} + \frac{1}{\mu_2} + \frac{1}{L_2} - \frac{1}{L_1} \big(2+\frac{L_2}{\mu_2}\big) \big]}{(L_1-\mu_1)(L_2-\mu_2)}
    \frac{\|G^+ - L_2 \Delta x\|^2}{2}.
\end{aligned}%
$$%
The weight of $\|G-L_2 \Delta x\|^2$ is positive ($\mu_1 \geq 0$), while the weight of $\|G^+-L_2 \Delta x\|^2$ is positive if either $\mu_2 {}\geq{} 0$ or $\mu_2 < 0$ and $\mu_1^{-1} + \mu_2^{-1} + L_2^{-1} - L_1^{-1}
            {}\leq{} 
        L_1^{-1} (1+\frac{L_2}{\mu_2})$. Under these two cases, by neglecting both mixed terms we get:%
$$
\begin{aligned}
\Delta F(x)
&\geq
    \frac{L_2-\mu_1}{L_2(L_1-\mu_1)}
    \frac{\|G\|^{2}}{2}
    + \frac{1}{L_2} \bigg(1 + \frac{L_2^{-1}-L_1^{-1}}{\mu_1^{-1}-L_1^{-1}}\bigg) \frac{\|G^+\|^2}{2},
\end{aligned}
$$
with equality only if $G = G^+ = L_2 \Delta x$.%

\smallskip  \textbf{Regime $p_3$:} It corresponds to $\mu_1 > 0$ and negativity of the weight of $\|G^+-L_2 \Delta x\|^2$ from the proof of $p_1$, i.e., $\mu_2 < 0$ and $\mu_1^{-1} + \mu_2^{-1} + L_2^{-1} - L_1^{-1}
            {}>{} 
        L_1^{-1} (1+\frac{L_2}{\mu_2})$. By setting $\alpha = \frac{-\mu_2}{L_2+\mu_2}$ in \eqref{eq:ineq_param_in_alpha}, after simplifications and building up a square including all weight of $L_2 \Delta x$ we get:%
$$
\hspace{-1em}
\begin{aligned}
\Delta F(x)
 {}\geq{}
    \frac{L_1^{-1} ( \mu_1^{-1}+\mu_2^{-1}+L_2^{-1} ) }{\mu_1^{-1}+\mu_2^{-1}+L_2^{-1}-L_1^{-1}} \frac{\|G\|^2}{2} +
    \frac{1}{\mu_2 + L_2} \frac{\|G^+\|^2}{2} +
    \\
    \frac{\mu_1 L_1 \mu_2 ( \mu_1^{-1}+\mu_2^{-1}+L_2^{-1}-L_1^{-1} )}{L_2(L_1-\mu_1)(L_2-\mu_2)}
    \frac{\big\| \frac{G}{\frac{L_1}{L_2} + \frac{\mu_2(L_1-\mu_1)}{\mu_1(L_2+\mu_2)}} - L_2 \Delta x\big\|^2}{2}.
\end{aligned}%
$$
Since $\mu_2 < 0$, the weights of $\|G\|^2$ and of the mixed term are both positive only if $\mu_1^{-1}+\mu_2^{-1}+L_2^{-1} \leq 0$, which is exactly condition \eqref{eq:threshold_cond_p3}. By neglecting the mixed term we get:%
$$
\hspace{-.5em}
\begin{aligned}
\Delta F(x)
 {}\geq{}&
    \frac{L_1^{-1} ( \mu_1^{-1}+\mu_2^{-1}+L_2^{-1} ) }{\mu_1^{-1}+\mu_2^{-1}+L_2^{-1}-L_1^{-1}} \frac{\|G\|^2}{2} {}+{}
    \frac{1}{\mu_2 + L_2} \frac{\|G^+\|^2}{2},
\end{aligned}%
$$%
with equality only if $G = L_2 \left[\frac{L_1}{L_2} + \frac{\mu_2(L_1-\mu_1)}{\mu_1(L_2+\mu_2)} \right] \Delta x$.%

\smallskip  \textbf{Regime $p_5$:} If neither of the necessary conditions for $p_1$ or $p_3$ are met, i.e., $ 0 {}<{} \mu_1^{-1} + \mu_2^{-1} + L_2^{-1}
    {} < {}
    L_1^{-1} (2 + \frac{L_2}{\mu_2})$ and $\mu_2 < 0$, we set $\alpha = \frac{\mu_1+\mu_2}{-\mu_2} > 0$ in \eqref{eq:ineq_param_in_alpha}. After simplifications and building up the squares we get:%
$$
        \begin{aligned}
           \Delta F(x)
 {}\geq{}
    \frac{\mu_1 + \mu_2}{\mu_2^2}
    \frac{\|G^+\|^{2}}{2} {} + {} 
    \frac{1}{L_1-\mu_1}
    \frac{\|  G - \mu_1 \Delta x \|^2}{2}
    {}+{} \\    
    \frac{-\mu_2^3 \mu_1 L_2  (\mu_1^{-1} + \mu_2^{-1} + L_2^{-1} )}{L_2-\mu_2}
    \frac{\| G^+ - \mu_2 \Delta x \|^2}{2},
        \end{aligned}%
$$
where the mixed terms have positive weights and therefore can be disregarded to obtain:%
$$
           \Delta F(x)
 {}\geq{}
    \frac{\mu_1 + \mu_2}{\mu_2^2}
    \frac{\|G^+\|^{2}}{2},
$$%
holding with equality only if $G = \mu_1 \Delta x$ and $G^+ = \mu_2 \Delta x$.%

\smallskip  \textbf{Regime $p_7$:} Assume $L_2 < \mu_1$ (i.e., $F$ is strongly convex) and $\mu_2 ( \mu_1^{-1} + \mu_2^{-1} + L_2^{-1} ) \geq 0$, with $\mu_2 \geq 0$ or $\mu_1 > -\mu_2 > 0$. By setting $\alpha = \frac{\mu_1}{L_2}$ in \eqref{eq:ineq_param_in_alpha}, after simplifications and building up the squares we get:%
$$
        \begin{aligned}
    \Delta F(x)
        {}\geq{}
    & \frac{L_2 + \mu_1}{L_2^2}
   \frac{\|G^+\|^{2}}{2} +
   \frac{1}{L_1-\mu_1}
    \frac{\|G - \mu_1 \Delta x\|^2}{2} + \\
    & 
    {}+{}
    \frac{\mu_1 \mu_2 (\mu_1^{-1} + \mu_2^{-1} + L_2^{-1})}{L_2(L_2-\mu_2)}
    \frac{\| G^+ - L_2 \Delta x \|^2}{2},
        \end{aligned}%
$$%
where the mixed term can be neglected and obtain:%
$$
    \Delta F(x)
        {}\geq{}
\frac{L_2 + \mu_1}{L_2^2}
   \frac{\|G^+\|^{2}}{2},
$$%
holding with equality only if $G = \mu_1 \Delta x$ and $G^+ = L_2 \Delta x$.%
\end{proof}%

\subsection{Proof of \cref{thm:dca_rates_N_steps}}

\begin{proof}
From \cref{thm:one_step_decrease_dca}, by taking the minimum between the (sub)gradients difference norms in \eqref{eq:tight_dist_1_with_sigmas}:%
$$
\begin{aligned}    
    F(x)-F(x^+) & {}\geq{} \sigma_i \frac{1}{2} \|g_1-g_2\|^2 + \sigma_i^+ \frac{1}{2} \|g_1^+ - g_2^+ \|^2 \\
               & {}\geq{} p_i \frac{1}{2} \min \{ \|g_1-g_2\|^2,\|g_1^+ - g_2^+ \|^2 \},
\end{aligned}%
$$%
where $p_i = \sigma_i + \sigma_i^+$, for $i=1,\dots,8$, as given in \cref{tab:DCA_regimes_one_step}. The rate \eqref{eq:rate_N_steps_no_F*} results by telescoping the above inequality for $N$ iterations and taking the minimum among all (sub)gradients differences norms. A rate with respect to $F^*$ is obtained either by applying the trivial bound $F(x^0)-F(x^N) \geq F(x^0)-F^*$ or, if $L_1 > \mu_2$, by using a tighter bound like the one demonstrated in \cite[Lemma 2.1]{abbaszadehpeivasti2021_DCA}:
    $$
        F(x^N)-F^* {}\geq{} \frac{1}{2(L_1 - \mu_2)} \|g_1^N-g_2^N\|^2.
    $$
    By incorporating the later to the telescoped sum and once again taking the minimum among all (sub)gradients differences norms we obtain the rate in \eqref{eq:rate_N_steps_with_F*}.
\end{proof}%

\section{Rates for the nonsmooth case (NS)}\label{sec:both_nonsmooth}
Within this section we assume $L_1 = L_2 = \infty$ and introduce the following metric to analyze the progress of \eqref{eq:DCA_it}:
\begin{align}\label{eq:new_convergence_criterion}\tag{$T(x)$}
\begin{aligned}
    T(x) {}:={}& f_1(x) - f_2(x) - \inf_w \{ f_1(w) - f_2(x) - \langle g_2\,,\, w - x \rangle \} \\
        {}={}& f_1(x) - f_1(x^+) - \langle g_1^+\,,\, x-x^+ \rangle,
\end{aligned}
\end{align}%
where $g_2 \in \partial f_2(x)$ and $g_1^+ \in \partial f_1(x^+)$, the second identity resulting from the \eqref{eq:DCA_it} iteration definition. This measure is used for studying the convergence of algorithms on nonconvex problems such as proximal gradient methods \cite[Theorem 5]{Karimi_condition_Func_minus_Envelope_2016} or the Frank-Wolfe algorithm \cite[Equation (2.6)]{Ghadimi2019}. In the context of DCA, \eqref{eq:new_convergence_criterion} is applied when both $f_1$ and $f_2$ are convex in \cite[Section 4]{abbaszadehpeivasti2021_DCA}. We extend their analysis to encompass hypoconvex functions, also implicitly providing a simpler and stronger proof for the convex case. %

If $\mu_1 \geq 0$ or $\mu_2 \geq 0$, it holds $T(x) \geq 0$: when $\mu_2 \geq 0$ this is implied by the subgradient inequality on $f_2$ and the first line of the definition \eqref{eq:new_convergence_criterion}; similarly, when $\mu_1 \geq 0$, it comes from the subgradient inequality on $f_1$ and second line of \eqref{eq:new_convergence_criterion}. Furthermore, if $T(x^{k}) = 0$ is achieved for some iterate $x^k$, then $F(x^{k+1})=F(x^k)$ and $x^k$ is a critical point. %
\smallskip \begin{theorem}\label{thm:both_functions_nonsmooth} 
Let $f_1 \in \mathcal{F}_{\mu_1,\infty}$ and $f_2 \in \mathcal{F}_{\mu_2,\infty}$, with $\mu_1 + \mu_2 > 0$ or $\mu_1=\mu_2=0$. One iteration of \eqref{eq:DCA_it} satisfies
\begin{align}\label{eq:bound_on_T_for_mu2_geq_0}
    \mu_2 \tfrac{\|x-x^+\|^2}{2} + T(x) 
    {}\leq{} F(x) - F(x^+), {}\,{} \text { if  $\mu_2 \geq 0$; } \\
\label{eq:bound_on_T_for_mu2_<_0}
    \tfrac{\mu_1+\mu_2}{\mu_1} {}\,{} T(x) 
    {}\leq{} F(x) - F(x^+),   \text { if  $\mu_2 < 0$.}
\end{align}
\end{theorem}%
\medskip\begin{proof} %
The subgradient inequality of $f_2$ reads:
    \begin{align}\label{eq:strg_cvx_f2_(x+,x)}
        f_2(x^+) \geq f_2(x) + \langle g_2 , x^+ - x \rangle + \mu_2 \tfrac{\|x-x^+\|^2}{2}.
    \end{align}
\noindent \textbf{Case $\mu_2 \geq 0$.} By adding and subtracting $f_1(x) - f_1(x^+)$ in both sides of \eqref{eq:strg_cvx_f2_(x+,x)} we obtain exactly \eqref{eq:bound_on_T_for_mu2_geq_0}:
    $$
       F(x) - F(x^+)
        \geq f_1(x) - f_1(x^+) + \langle g_2 , x^+ - x \rangle + \mu_2 \tfrac{\|x-x^+\|^2}{2}.
    $$
\noindent \textbf{Case $\mu_2 < 0$.} We have $\mu_1 > -\mu_2 > 0$. From the strong convexity of $f_1$, with $g_1^+ \in \partial f_1(x)$:
$$
    f_1(x) \geq f_1(x^+) - \langle g_1^+ , x^+ - x \rangle + \mu_1 \frac{\|x-x^+\|^2}{2}.
$$
By multiplying it with $-\frac{\mu_2}{\mu_1} > 0$ and summing with \eqref{eq:strg_cvx_f2_(x+,x)} we get, using $g_1^+ = g_2$:
$$
    -\tfrac{\mu_2}{\mu_1}(f_1(x) - f_1(x^+)) {}-{}
    f_2(x) + f_2(x^+)
        {}\geq{}
    (1+\tfrac{\mu_1}{\mu_2}) \langle g_2 , x^+ - x \rangle.
$$
By adding $(1+\tfrac{\mu_2}{\mu_1})(f_1(x) - f_1(x^+))$ in both sides:
$$
    F(x) - F(x^+)
        {}\geq{}
    (1 + \tfrac{\mu_1}{\mu_2})
    \left( f_1(x) - f_1(x^+) + \langle g_2 , x^+ - x \rangle \right),
$$
which is exactly \eqref{eq:bound_on_T_for_mu2_<_0}.
\end{proof}%

\medskip \begin{theorem}\label{thm:convergence_both_nonsmooth}
Let $f_1 \in \mathcal{F}_{\mu_1,\infty}$ and $f_2 \in \mathcal{F}_{\mu_2,\infty}$, with $\mu_1 + \mu_2 > 0$ or $\mu_1=\mu_2=0$. Consider $N$ iterations of \eqref{eq:DCA_it} starting from $x^0$. Then, if $\mu_2 \geq 0$, the quantity
    \[
        \min_{0 \leq k \leq N-1} \{ T(x^k) \} + \frac{\mu_2}{2} \min\limits_{0 \leq k \leq N-1} \{ \|x^k - x^{k+1}\|^2 \}
    \]
is bounded from above by 
\[ \frac{F(x^0)-F^*}{N},\]
        whereas if $\mu_2 < 0$ we have 
    $$
        \min_{0 \leq k \leq N-1} \{ T(x^k) \}
            {}\leq{}
        \frac{\mu_1}{\mu_1+ \mu_2}
        \frac{F(x^0)-F^*}{N}.
    $$
\end{theorem}%

\medskip\begin{proof}
    By telescoping the inequalities from \cref{thm:both_functions_nonsmooth}, taking the minimum over all $T(x^k)$ and using the trivial bound $F(x^N) \geq F^*$.
\end{proof}%

\smallskip
Note that previous work in \cite{abbaszadehpeivasti2021_DCA} (in the case $\mu_2 \ge 0$) only had the first term $\min_{0 \leq k \leq N-1} \{ T(x^k) \}$ appearing in the left-hand side of the first inequality above.

\section{Conclusion}

In this work we thoroughly examined the behavior of one \eqref{eq:DCA_it} iteration applied to the DC framework extended to accommodate one hypoconvex function. We characterized eight distinct regimes for the objective  decrease in terms of subgradient differences. We conjectured, based on numerical observations, that certain regimes remain tight over multiple iterations. We also identified the asymptotic behavior in the usual nonconvex-nonconcave setting of the objective function.

In cases where the objective is nonsmooth, we divide the analysis in two parts. If one function is smooth, the convergence rate results as a corollary of the previous analysis. When both functions are nonsmooth, we demonstrate convergence using another appropriate measure. 

The structure of the \eqref{eq:DC_program} objective allows to easily shift a certain amount of curvature from one term to the other, by adding/subtracting $\frac{\rho}{2} \|x\|^2$ to each term. 
Our extension allowing one hypoconvex term in \eqref{eq:DC_program} allows extra flexibility in the splitting. Notably, applying the \eqref{eq:DCA_it} iteration on a split with one hypoconvex function yields different rates then applying iterations on a modified split with shifted curvatures achieving convexity in both functions. Further research may identify which split achieves the best convergence rate.

\addtolength{\textheight}{-0cm}   


\bibliographystyle{IEEEtran}
\bibliography{ms}

\end{document}